\documentclass[reqno]{amsart}
\usepackage{amsmath}
\usepackage{amssymb}
\usepackage{amsthm}
\usepackage{graphics}
\usepackage{rotating}
\usepackage{cite}
\usepackage{amsfonts}
\usepackage{color}

\newcommand{\field}[1]{\mathbb{#1}}
\newtheorem{definition}{Definition}

\newtheorem{theorem}{Theorem}
\newtheorem{proposition}{Proposition}

\newtheorem{corollary}{Corollary}
\title[On calculus with a quaternionic variable and its characteristic...]{On calculus with a quaternionic variable and its characteristic Cauchy-Riemann type equations}
\author{Daniel Alay\'{o}n-Solarz}
\address{Departamento de Matem\'aticas Puras y Aplicadas\\
Universidad Sim\'on Bol\'ivar \\
Caracas \\
Venezuela.
}
\email{danieldaniel@gmail.com}
\begin{document}
\maketitle

\begin{abstract}
We show sufficient and necessary conditions, in terms of partial differential equations with variable coefficients, for a quaternionic function to admit a continuous derivative in a open set in the sense of C. Schwartz.
\end{abstract}

\section{Introduction}

Let $\field{H}$ denote the quaternions , let $q = t + xi +yj +zk$ be a quaternion written as $q = t + r \iota$, 
where $t$ is the real part of the quaternion, $r \geq 0$   and $\iota^2 = -1$. The problem
of the quaternionic left-derivative starts by considering the limit
\begin{displaymath}
\lim_{h \to 0} h^{-1}(f(q+h) - f(q) - hA) = 0
\end{displaymath} 
for some number $A$. It is well known that the only functions left-derivable in $\field{H}$ in the above sense turn out to be $f(q) = qa + b$
for some $a,b \in \field{H}$. 
A new approach for this problem, which results in a larger class of functions has been introduced in \cite{CS}. 

\begin{definition}
Let $f$ be a quaternionic function that is continuous and real differentiable at $q$.
We say $f$ is
\textit{S-derivable by the left at $q$} if there exists quaternions $A$ or $B$ and $C$ such that the limit
\begin{displaymath} 
\lim_{h \to 0}
\begin{cases} 
 h^{-1} (f(q+ h) - f(q) - h A)   &\text{if } q \in \field{R} \\
 h^{-1} (f(q+ h) - f(q) - h_\parallel B -  h_\perp C)  &\text{if } q \in \field{H} \setminus \field{R}
\end{cases}
\end{displaymath} 
where
\begin{displaymath}
h_\parallel := \frac{h - \iota h \iota}{2}, \ \  h_\perp := \frac{h + \iota h \iota}{2}
\end{displaymath}
vanishes at $q$. Then we say the function $f$ is S-derivable at $q$. The numbers $A$ and $B$ we shall call the \textit{parallel derivative} and C the \textit{perpendicular derivative} at $q$.
\end{definition}
Just like in the complex setting we would like to consider functions that are derivable in this new, broader sense in some open set $\Omega \subseteq \field{H}$.
\begin{definition}
Let $f: \Omega \subseteq \field{H} \to \field{H}$ be a quaternionic function that is continuously real-differentiable. We say $f$ is
\textit{ continuosly S-derivable by the left in $\Omega$} if there exists continuous quaternionic functions $A(q): \Omega \cap \field{R} \to \field{H}$ and  $B(q),C(q): \Omega \setminus \field{R} \to \field{H}$ such that 
\begin{displaymath} 
\lim_{h \to 0}
\begin{cases} 
 h^{-1} (f(q+ h) - f(q) - h A(q)) = 0   &\text{if } q \in \field{R} \\
 h^{-1} (f(q+ h) - f(q) - h_\parallel B(q) -  h_\perp C(q)) = 0 &\text{if } q \in \field{H} \setminus \field{R}
\end{cases}
\end{displaymath}
for all $q \in \Omega \subseteq \field{H}$.
\end{definition}
In the complex case one obtains characteristic equations, namely the Cauchy-Riemann equations by evaluation of the derivative with two
prescribed directions, the real and the pure imaginary.
The result we consider first that this class can be characterized by \textit{some} partial differential equations, just like in the commutative complex
setting. More precisely:

\begin{theorem}
A function quaternionic function $f$ whose partial derivatives are continous in some open set  $\Omega$, then $f$  is  continuously 
S-derivable in $\Omega$ if, and only
if,  it satisfies the following partial differential equations in 
$\Omega \cap  \field{R}$
\begin{equation}
\frac{\partial f}{\partial x} = i \frac{\partial f}{\partial t}
\end{equation}
\begin{equation}
\frac{\partial f}{\partial y} = j \frac{\partial f}{\partial t}
\end{equation}
\begin{equation}
\frac{\partial f}{\partial y} = k \frac{\partial f}{\partial t}
\end{equation}
and the following partial differential equations in $\Omega \setminus  \field{R}$

\begin{equation}
\frac{\partial f}{\partial x} = \Big( \frac{i - \iota i \iota}{2} \Big) \frac{\partial f}{\partial t}   \ - \  \Big( \frac{i + \iota i \iota}{4} \Big) Df 
\end{equation}
\\
\begin{equation}
\frac{\partial f}{\partial y} = \Big( \frac{j - \iota j \iota}{2} \Big) \frac{\partial f}{\partial t}   \ - \  \Big( \frac{j + \iota j \iota}{4} \Big) Df
\end{equation}
\\
\begin{equation}
\frac{\partial f}{\partial z} = \Big( \frac{k - \iota k \iota}{2} \Big) \frac{\partial f}{\partial t}   \ - \  \Big( \frac{k + \iota k \iota}{4} \Big) Df 
\end{equation}
where D denotes the left Fueter operator:
\begin{displaymath}
Df := \frac{\partial f}{\partial t} +  i \frac{\partial f}{\partial x}+  j \frac{\partial f}{\partial y}+  k \frac{\partial f}{\partial z}
\end{displaymath}
\end{theorem}
This equations clearly play here the role of the Cauchy-Riemann equations played in the commutative setting. 
We then show that if a function is S-derivable then is necessarily Cullen-regular (also termed Slide regular) in the sense of of G. Gentili and D.C Struppa.
Cullen-regular functions have been shown, see \cite{GG} to be expressible as absolutely convergent  quaternionic unilateral power series assuming
the domain of the function is a ball centered at the origin with radius $R$. One consequence of the previous result is that it allows to state the following 
characterization of a continuously S-derivable function in a ball of radius $R$ centered at the origin.
\begin{theorem}
Let $f: B(0,R) \to \field{H}$ be a $C^{1}$ quaternionic function. Then $f$  is S-derivable by the left if and only if 
it is expressible as a absolutely
convergent power series
\begin{displaymath}
f(q) = \sum_{k = 0}^{\infty} q^k a_k, \ a_k \in \field{H}
\end{displaymath}
and
\begin{displaymath}
A(q) = \frac{\partial f}{\partial t}, \ q \in B(0,R) \cap \field{R}
\end{displaymath}
\begin{displaymath}
B(q) = \frac{\partial f}{\partial t}, \ q \in B(0,R) \cap \field{H} \setminus \field{R}
\end{displaymath}
\begin{displaymath}
C(q)= (q- \overline q)^{-1} [f(q)-f(\overline{q})]
\end{displaymath}
\end{theorem}
Thus recovering the original result of C. Schwartz.
\section{Necessary condition for S-differentiability using directions based on $1,i,j$ and $k$}
In complex calculus the necessity of the Cauchy-Riemann equations is usually obtained by considering the limit that
defines the derivative in the real and pure imaginary directions at a point.  Proceeding analogously with the quaternionic
case one obtains corresponding identities.
\begin{proposition} 
Let $f$ be continuously S-differentiable in $\Omega$, then 
\begin{displaymath}
A(q) = \frac{\partial f}{\partial t}(q), q \in \Omega 
\end{displaymath}
\end{proposition}
\begin{proof}  By the nature of the definition we must check the cases when the base point $q$ is real or not. Let $q$ be real and $f$ S-derivable at $q$. Choosing $h$ as a real number and after
taking the limit as $h \to 0$ we obtain
\begin{displaymath}
A(q) = \frac{\partial f}{\partial t}(q)
\end{displaymath} 
Now assume $q$ is not real. Choose again $h$ a real number. Since $h$ is real then $h_\parallel = h $ and $h_\perp=0$, so if $f$
is S-differentiable we have
\begin{displaymath}
B(q) = \frac{\partial f}{\partial t}(q)
\end{displaymath}
for all $q \in \Omega \setminus \field{R}$.
\end{proof}
\begin{proposition} 
Let $f$ be continuously S-differentiable, then $f$ satisfies equations (1) to (6) and 
\begin{displaymath}
C(q) = -(\frac{1}{2}) Df(q)
\end{displaymath}
where $D$ is the left Fueter operator.
\end{proposition}
\begin{proof}

The preceeding proposition means $A(q)$ and $B(q)$ can be replaced by $\frac{\partial f}{\partial t}$ without loss of generality. Again we start with $q$ real.
Choosing $h$ as $\epsilon i$, $\epsilon j$ and $\epsilon k$  where $\epsilon$ is real and taking the corresponding limit as $\epsilon$ goes to zero 
we obtain necessary conditions for the partial derivatives with respect to $x,y$ and $z$. If $q$ is real we obtain immediately equations (1), (2) and (3). 
If $q$ is non-real then we obtain the following equations, with $C(q)$ a function to be determined:
\begin{equation}
\frac{\partial f}{\partial x} = \Big( \frac{i - \iota i \iota}{2} \Big) \frac{\partial f}{\partial t}   \ + \  \Big( \frac{i + \iota i \iota}{2} \Big) C(q), 
\end{equation}
\\
\begin{equation}
\frac{\partial f}{\partial y} = \Big( \frac{j - \iota j \iota}{2} \Big) \frac{\partial f}{\partial t}   \ + \  \Big( \frac{j + \iota j \iota}{2} \Big) C(q),
\end{equation}
\\
\begin{equation}
\frac{\partial f}{\partial z} = \Big( \frac{k - \iota k \iota}{2} \Big) \frac{\partial f}{\partial t}   \ + \  \Big( \frac{k + \iota k \iota}{2} \Big) C(q), 
\end{equation}
\\
For this we multiply Eq. (7) by the left by $i$, Eq. (8) by the left by $j$ and Eq. (9) by the left by $k$ and sum all three
resulting equations. For the left hand side one thus obtains:
\begin{displaymath}
i \frac{\partial f}{\partial x} + j \frac{\partial f}{\partial y} +  k \frac{\partial f}{\partial y}.
\end{displaymath}
Thhe right hand side is further simplified with the following identities:
\begin{displaymath}
 i \Big(  \frac{i - \iota i \iota}{2} \Big) + j  \Big( \frac{j - \iota j \iota}{2} \Big) +  k \Big( \frac{k - \iota k \iota}{2} \Big) = -1,
\end{displaymath}
and
\begin{displaymath}
 i \Big(  \frac{i + \iota i \iota}{2} \Big) + j  \Big( \frac{j + \iota j \iota}{2} \Big) +  k \Big( \frac{k + \iota k \iota}{2} \Big) = -2.
\end{displaymath}
Thus one obtains
\begin{displaymath}
i \frac{\partial f}{\partial x} + j \frac{\partial f}{\partial y} +  k \frac{\partial f}{\partial z} = -\frac{\partial f}{\partial t} -2 C(q)
\end{displaymath}
or
\begin{displaymath}
\frac{\partial f}{\partial t} + i \frac{\partial f}{\partial x} + j \frac{\partial f}{\partial y} +  k \frac{\partial f}{\partial z} =  -2 C(q)
\end{displaymath}
\end{proof}

\section{Sufficient conditions for the S-derivability}

\begin{proposition} Let $f : \Omega \to \field{H}$ be continuously  real-differentiable and suppose that f satisfies Eq. (1) to (6) in
$\Omega$, then f is S-derivable in $\Omega$.
\end{proposition}
\begin{proof} Because $f$ is continuosly real differentiable in $\Omega$ then:
\begin{displaymath}
\lim_{h \to 0} \frac{| f(q+h)-f(q) - M_fh |}{|h|} = 0
\end{displaymath}
where $M_f$ is the Jacobi matrix of $f$ at $q$. Now writing $h = h_0 + i h_1 + j h_2 + k h_3$, $h_n \in \field{R}$ we can also write
\begin{displaymath}
M_f h = \frac{\partial f}{\partial t} h_0 +  \frac{\partial f}{\partial y} h_1 + \frac{\partial f}{\partial y} h_2
+ \frac{\partial f}{\partial z} h_3
\end{displaymath}
If $q$ is real one starts with
and use Eq. (1), (2) and (3),
to  obtain
\begin{displaymath}
M_f h = h \frac{\partial f}{\partial t}.
\end{displaymath}
For $q$ non real, and if $f$ satisfies equations (4),(5) and (6) we have
\begin{displaymath}
\frac{\partial f}{\partial t} h_0 +  \frac{\partial f}{\partial y} h_1 + \frac{\partial f}{\partial y} h_2
+ \frac{\partial f}{\partial z} h_3 = 
\end{displaymath}
\begin{displaymath}
\frac{\partial f}{\partial t} h_0 + \Big( \frac{i h_1 - \iota i h_1 \iota}{2} \Big) \frac{\partial f}{\partial t}   \ - \  \Big( \frac{i h_1 + \iota i h_1 \iota}{4} \Big) Df 
\end{displaymath}
\begin{displaymath}
+  \Big( \frac{j h_2 - \iota j h_2 \iota}{2} \Big) \frac{\partial f}{\partial t}   \ - \  \Big( \frac{j h_2 + \iota j  h_2 \iota}{4} \Big) Df
\end{displaymath}
\begin{displaymath}
+ \Big( \frac{ k h_3 - \iota k h_3 \iota}{2} \Big) \frac{\partial f}{\partial t}   \ - \  \Big( \frac{k h_3 + \iota k h_3 \iota}{4} \Big) Df 
\end{displaymath}
which can be rewritten in the compact form
\begin{displaymath}
M_f h = h_{\parallel} \frac{\partial f}{\partial t} + h_{\perp} (\frac{-1}{2} Df)
\end{displaymath}
therefore we conclude that for $q$ real:
\begin{displaymath}
\lim_{h \to 0} |  h^{-1}(f(q+h)-f(q) - h \frac{\partial f}{\partial t})  |= \lim_{h \to 0}  \frac{|  f(q+h)-f(q) -M_f h |}{| h |} = 0
\end{displaymath} 
which implies 
\begin{displaymath}
\lim_{h \to 0}    h^{-1}(f(q+h)-f(q) - h \frac{\partial f}{\partial t})    = 0
\end{displaymath}
and similarly, for a $q$ non-real:
\begin{displaymath}
\lim_{h \to 0} |  h^{-1}(f(q+h)-f(q) - h_{\parallel} \frac{\partial f}{\partial t} + h_{\perp} \frac{1}{2} Df )  |= \lim_{h \to 0}  \frac{|  h(f(q+h)-f(q) -M_f h |}{| h |} = 0
\end{displaymath} 
which implies 
\begin{displaymath}
\lim_{h \to 0}   h^{-1}(f(q+h)-f(q) - h_{\parallel} \frac{\partial f}{\partial t} + h_{\perp} \frac{1}{2} Df )  = 0
\end{displaymath}
and therefore we have found the $A(q), B(q)$ and $C(q)$ required for this function $f$ to be continuosly S-derivable.
\end{proof}
 
\section{Cullen-regularity and commuting directional derivatives}
The above section shows that the four directions $1,i, j$ and $k$ are sufficient and necessary for finding  characteristic equations. However $i, j$ and $k$ 
are mere examples of $\sqrt{-1}$ in $\field{H}$. Let $\field{S}^2 \subseteq \field{H}$ be the set of roots of $-1$, and fix $I \in \field{S}^2$. Since $I^2 = -1$
the set $\field{R} + I \field{R}$  is a isomorphic copy of $\field{C}$. Now let $f$ be
continuously  S-derivable in some $\Omega$ whose intersection with the real numbers is non-empty. For a real $q$ we consider
\begin{displaymath}
\lim_{I \epsilon \to 0} (I \epsilon)^{-1} (f(q + \epsilon I) -f(q)) = \frac{\partial f}{\partial t}
\end{displaymath}
on the other side, interpreting the left-hand side as a directional derivative:
\begin{displaymath}
\lim_{\epsilon \to 0} \frac{f(q + \epsilon I) -f(q)}{\epsilon} = \frac{\partial f}{\partial r}
\end{displaymath}
when restricted to $\field{R} + I \field{R}$. Thus we conclude that for $q$ a real number the expression
\begin{displaymath}
 \frac{\partial f}{\partial t} + I \frac{\partial f}{\partial r} = 0
\end{displaymath}
vanishes in $\field{R} + I \field{R}$. If $q \in \field{H} \setminus \field{R}$, write it as $q = t + r I$, for some $I ^2 = -1$ and $r \geq 0$. Consider increments only  in the imaginary component of $\field{R} + I \field{R}$:
$h = \epsilon I$, then $h_\parallel = \epsilon I$, $h_\perp = 0$. Then
\begin{displaymath}
\frac{\partial f}{\partial r} = \lim_{\epsilon \to 0} \frac{f(q + \epsilon I) -f(q)}{\epsilon} = I \frac{\partial f}{\partial t} 
\end{displaymath}
and, as $I^2 = -1$:
\begin{displaymath}
\frac{\partial f}{\partial t} + I \frac{\partial f}{\partial r} = 0
\end{displaymath}
at $q$. Let $\iota: \field{H} \setminus \field{R} \to \field{S}^2 \subseteq \field{H} \setminus \field{R}$ be the function that sends a quaternion $q = t + r I$ to $I$. Then the above means that $f$ satifies 
\begin{displaymath}
(\frac{\partial }{\partial t} + \iota \frac{\partial }{\partial r})f = 0.
\end{displaymath}
This discussion shows that S-derivable functions are necessarily Cullen-regular in the sense of G. Gentili and D.C Struppa.
\section{The iota-derivative and the perpendicular derivative}
The previous section only considered increments that commuted with the variable $q$. Start parametrizing
the function $\iota$ as $\cos \alpha \sin \beta i + \sin \alpha \sin \beta j + \cos \beta$
writing the left Fueter operator as 
\begin{displaymath}
D f = \frac{\partial f}{\partial t} + \iota \frac{\partial f}{\partial r} - \frac{1}{r} \frac{\partial f}{\partial \iota}
\end{displaymath}
where
\begin{displaymath}
\frac{\partial f}{\partial \iota} :=( \frac{\partial \iota}{\partial \alpha})^{-1} \frac{\partial}{\partial \alpha} +  (\frac{\partial \iota}{\partial \beta})^{-1} \frac{\partial}{\partial \beta}
\end{displaymath}
Which is defined outside the singular subplane $\field{R} + k \field{R}$.
observe that $( \frac{\partial \iota}{\partial \alpha})^{-1}$ and $( \frac{\partial \iota}{\partial \beta})^{-1}$ anti-commute with the function $\iota$, so they are geometrically
perpendicular to $\iota$. For any continuously real-derivable Cullen-regular function $f$ we can write
\begin{displaymath}
Df = -\frac{1}{r} \frac{\partial f }{\partial \iota}
\end{displaymath}
so the perpendicular derivative can also be written, without loss of generality as
\begin{displaymath}
 \frac{1}{2r} \frac{\partial f }{\partial \iota}.
\end{displaymath}

\section{A criterium for the S-derivability}
\begin{theorem}
Let $f : \Omega \to \field{H}$ be a real-derivable quaternionic function such that 
\begin{itemize}
\item $f$ is S-derivable in $\Omega \cap \field{R}$.
\item $f = u+ \iota v$ for some  $C^1$ quaternionic functions $u, v$.
\item $u$ and $v$ are independent of $\alpha$ and $\beta$.
\item $u$ and $v$ satisfy:
\begin{displaymath}
\frac{\partial u}{\partial t} - \frac{\partial v}{\partial r} = \frac{\partial u}{\partial r} + \frac{\partial v}{\partial t} = 0
\end{displaymath} 
$\Omega \setminus \field{R}$
\end{itemize}
then with this conditions $f$ is continuously left S-derivable in $\Omega$ and its perpendicular derivative is
$\dfrac{v}{r}$.
\end{theorem}
\begin{proof}
Let $f$  satisfy the hypothesis. We show $f$ is continuosly S-derivable in $\Omega \setminus \field{R}$. Computing
the partial derivatives:
\begin{displaymath}
\frac{\partial f}{\partial x} = \frac{\partial u}{\partial x}  + \iota_x v + \iota \frac{\partial v}{\partial x} 
\end{displaymath}
\begin{displaymath}
\frac{\partial f}{\partial y} = \frac{\partial u}{\partial y} + \iota_y v + \iota \frac{\partial v}{\partial y} 
\end{displaymath}
\begin{displaymath}
\frac{\partial f}{\partial z} = \frac{\partial u}{\partial z}  + \iota_z v + \iota \frac{\partial v}{\partial z} 
\end{displaymath}
on the other hand, since neither $u$ nor $v$ depend on $\alpha$ and $\beta$ implies that
\begin{displaymath}
\frac{\partial u}{\partial x} = \frac{\partial u}{\partial r} \frac{x}{r}, \ \frac{\partial v}{\partial x} = \frac{\partial v}{\partial r} \frac{x}{r}
\end{displaymath}
\begin{displaymath}
\frac{\partial u}{\partial y} = \frac{\partial u}{\partial r} \frac{y}{r}, \ \frac{\partial v}{\partial y} = \frac{\partial v}{\partial r} \frac{y}{r}
\end{displaymath}
\begin{displaymath}
\frac{\partial u}{\partial z} = \frac{\partial u}{\partial r} \frac{z}{r}, \ \frac{\partial v}{\partial z} = \frac{\partial v}{\partial r} \frac{z}{r}
\end{displaymath}
thus we can write
\begin{displaymath}
\frac{\partial f}{\partial x} =  \frac{x}{r} \frac{\partial f}{\partial r} + \iota_x v
\end{displaymath}
\begin{displaymath}
\frac{\partial f}{\partial y} =  \frac{y}{r} \frac{\partial f}{\partial r} + \iota_y v
\end{displaymath}
\begin{displaymath}
\frac{\partial f}{\partial z} =  \frac{z}{r} \frac{\partial f}{\partial r} + \iota_z v
\end{displaymath}
the hypothesis on $f$ imply that we can replace $\frac{\partial f}{\partial r}$ with $\iota \frac{\partial f}{\partial t}$.  
\begin{displaymath}
\frac{\partial f}{\partial x} =  \frac{x}{r}  \iota \frac{\partial f}{\partial t} + \iota_x v
\end{displaymath}
\begin{displaymath}
\frac{\partial f}{\partial y} =  \frac{y}{r} \iota \frac{\partial f}{\partial t} + \iota_y v
\end{displaymath}
\begin{displaymath}
\frac{\partial f}{\partial z} =  \frac{z}{r} \iota \frac{\partial f}{\partial t} + \iota_z v
\end{displaymath}
finally we need the following identities:
\begin{displaymath}
\frac{x}{r} \iota =  \Big( \frac{i - \iota i \iota}{2} \Big), \  \frac{y}{r} \iota =  \Big( \frac{j - \iota j \iota}{2} \Big), \  \frac{z}{r} \iota =  \Big( \frac{k - \iota k \iota}{2} \Big)
\end{displaymath}
and
\begin{displaymath}
 \iota_x = \frac{1}{r} \Big( \frac{i + \iota i \iota}{2} \Big), \   \iota_y =   \frac{1}{r} \Big( \frac{j + \iota j \iota}{2} \Big), \  \iota_z =   \frac{1}{r}\Big( \frac{k + \iota k \iota}{2} \Big)
\end{displaymath}
Since the function is Cullen-regular and neither $u$ nor $v$ depend of $\alpha$ and $\beta$ we conclude
that
\begin{displaymath}
Df = -2 \frac{v}{r},
\end{displaymath}
so  $f$ satisfies equations (4),(5) and (6).
\end{proof}
\begin{corollary}
The function $\iota$ is continuously S-derivable in $\field{H} \setminus \field{R}$.
\end{corollary}
\begin{proof}
Take $u =0$ and $v =1$. 
\end{proof}
\begin{corollary}
The power function is S-derivable in $\field{H}$.
\end{corollary}
\begin{proof}
Let $f(q) =  q^n$ for some $n \in \field{N}$, $f$ is continuously real-derivable in $\field{H}$. Let $q$ be a real number. Then
\begin{displaymath}
\frac{\partial}{\partial t} q^n = n q^{n-1} 
\end{displaymath}
\begin{displaymath}
\frac{\partial }{\partial x} q^n = i n q^{n-1}
\end{displaymath}
\begin{displaymath}
\frac{\partial }{\partial y} q^n = j n q^{n-1}
\end{displaymath}
\begin{displaymath}
\frac{\partial }{\partial z} q^n = k n q^{n-1}
\end{displaymath}
and therefore $q^n$ is S-derivable in $\field{R}$. Since $q^n = u_n (t,r) + \iota v_n(t,r)$, for some real functions $u_n$ and $v_n$ satisfying
\begin{displaymath}
\frac{\partial u_n}{\partial t} - \frac{\partial v_n}{\partial r} = \frac{\partial u_n}{\partial r} + \frac{\partial v_n}{\partial t} = 0
\end{displaymath} 
we obtain the desired result.
\end{proof}
\begin{corollary}
Let
\begin{displaymath}
f(q) =  \sum_{n=0}^{\infty} q^n a_n
\end{displaymath}
be an absolutely and uniformly convergent series in a ball $B(0,R)$. Then $f$ is S-derivable in $B(0,R)$.
\end{corollary}
\begin{proof}
Observe such $f$ is real-derivable in $B(0,1)$, (in fact is real-analytic). Let $q \in B(0,R)$ be real. Uniform and absolute convergence means the series can be partial-derivated termwise.
Thus one obtains that $f$ satisfies  equations (1),(2) and (3) for every $q$ in $B(0,R) \cap \field{R}$. On the other side we have
\begin{displaymath}
\sum_{n=0}^{\infty} q^n a_n = \sum_{n=0}^{\infty} u_n a_n + \iota  \sum_{n=0}^{\infty} v_n a_n = u(t,r) + \iota v(t,r)
\end{displaymath}
and such $u$ and $v$ satisfy:
\begin{displaymath}
\frac{\partial u}{\partial t} - \frac{\partial v}{\partial r} = \frac{\partial u}{\partial r} + \frac{\partial v}{\partial t} = 0
\end{displaymath} 
\end{proof}
Observe that for functions constructed as in this section one has $f(\overline{q})= u(q) - \iota v(q)$. So they also must also satify the
following equalities:
\begin{displaymath}
\frac{1}{2r} \frac{\partial f }{\partial \iota} = \frac{v}{r} = (\iota r)^{-1} (\iota v) = (q - \overline{q})^{-1}(f(q) - f(\overline{q}))
\end{displaymath}
We are now ready to prove \textbf{Theorem 2}.
\begin{proof}
In \cite{GG} it is proved that every Cullen-regular functions $f :B(0,R) \to \field{H}$ can be expressed as such
power series. Since S-derivable functions are in particular Cullen-regular this theorem remains
valid. On the other side, functions expressible as such power series have been shown to be
S-derivable themselves. The identities $A(q) =  B(q) = \frac{\partial f}{\partial t}$ and $Df = -2  (q - \overline{q})^{-1}(f(q) - f(\overline{q}))$
holds for all functions constructed as in \textbf{Theorem 3} and such power series are of this form.
\end{proof}

%%%%%%%%%%%%%%%%%%%%%%%%%%%%%%%%%%%%%%%%%%%%%%%%%%%%%%%%%%%%%%%%%%%%%%%%%%%%%%%
%                                REFERENCES
%%%%%%%%%%%%%%%%%%%%%%%%%%%%%%%%%%%%%%%%%%%%%%%%%%%%%%%%%%%%%%%%%%%%%%%%%%%%%%%


\begin{thebibliography}{99}
\bibitem{GG}
G. Gentili, D.C. Struppa, \textit{A new theory of regular functions of a quaternionic variable}, Adv. 
Math. \textbf{216}, 279-301. (2007)
\bibitem{CS}
C. Schwartz, \textit{Calculus with a quaternionic variable} J. Math. Phys. \textbf{50}, 013523 (2009)
\end{thebibliography}
\end{document}